\documentclass[onefignum,onetabnum]{siamart190516}

\usepackage{graphics} 
\usepackage{epsfig} 
\usepackage{mathptmx} 
\usepackage{times} 
\usepackage{amsmath} 
\usepackage{amssymb,mathrsfs}  

\usepackage{enumitem}

\newtheorem{applemma}{Lemma A\hspace{-0.5mm}}[section]
{}

\newcommand{\enurom}[1]
{\begin{enumerate}[label=\rm{(\roman*)}]
		#1
\end{enumerate}}

\newcommand{\enualp}[1]
{\begin{enumerate}[label=(\alph*)]
		#1
\end{enumerate}}

\newcommand{\ccd}[0]
{(\cdot)}

\newcommand{\ttnn}[1]
{\textnormal{#1}}

\newcommand{\modulo}[1]
{\left|#1\right|}

\newcommand{\rif}[1]
{(\ref{#1})}

\newcommand{\norm}[1]
{\left\| #1 \right\|}

\newcommand{\ra}
{\rightarrow}

\newcommand{\ps}[2]
{\langle\,#1,#2\rangle}

\newcommand{\rt}[0]
{(t)}

\newcommand{\rs}[0]
{(s)}

\newcommand{\mineq}[1]
{\leq #1}

\newcommand{\mageq}[1]
{\geq #1}

\newcommand{\equazioneref}[2]
{
	\begin{align}\label{#1}\begin{split}
	#2
	\end{split}\end{align}
}

\newcommand{\sistemaref}[2]
{
	\begin{align}\label{#1}
	\begin{cases}
	#2
	\end{cases}
	\end{align}
}

\newcommand{\sistemanoref}[1]
{
	\begin{align*}
	\begin{cases}
	#1
	\end{cases}
	\end{align*}
}

\newcommand{\eee}[1]
{
	\begin{align*}
	#1
	\end{align*}
}

\newcommand{\graffe}[1]
{\left\{ #1 \right\}}

\newcommand{\eps}[0]
{\varepsilon }

\newcommand{\ccal}[1]
{\mathcal{#1}}

\newcommand{\scr}[1]
{\mathscr{#1}}

\newcommand{\bb}[1]
{\mathbb{#1}}

\DeclareMathOperator*{\Liminf}{Lim\,inf}
\DeclareMathOperator*{\Limsup}{Lim\,sup}

\newcommand{\rpiu}[0]
{ \mathbb R^+ }

\newcommand{\rzero}[0]
{ \mathbb R^+ }

\usepackage{lipsum}
\usepackage{amsfonts}
\usepackage{graphicx}
\usepackage{epstopdf}
\usepackage{algorithmic}
\ifpdf
\DeclareGraphicsExtensions{.eps,.pdf,.png,.jpg}
\else
\DeclareGraphicsExtensions{.eps}
\fi

\newsiamremark{remark}{Remark}
\newsiamremark{hypothesis}{Hypothesis}
\crefname{hypothesis}{\textnormal{H.}\hspace{-0.5mm}}{Hypotheses}
\newsiamthm{claim}{Claim}

\headers{Two-player representation, infinite horizon, and state constraints}{V. Basco and P. M. Dower}

\title{A two-player game representation for a class of infinite horizon control problems under state constraints\thanks{{A preliminary version of this manuscript appears in arXiv:1912.08554 and an abbreviated version was accepted to be published in the proceedings of the 2020 American Control Conference. This research is partially supported by AFOSR grant FA2386-16-1-4066.}}}

\author{Vincenzo Basco\thanks{Basco (corresponding author) and Dower are with the Department of Electrical \& Electronic Engineering, University of Melbourne, Victoria 3010, Australia. Email {\tt\{vincenzo.basco, pdower\}@unimelb.edu.au}.}
	\and Peter M. Dower\footnotemark[2]
}

\usepackage{amsopn}

\ifpdf
\hypersetup{
	pdftitle={A two-player game representation for a class of infinite horizon control problems under state constraints},
	pdfauthor={Vincenzo Basco and Peter M. Dower}
}
\fi

\begin{document}
	
	\renewcommand\ttdefault{cmvtt}

	\maketitle

	\begin{abstract}
		In this paper feedback laws for a class of infinite horizon control problems under state constraints are investigated. We provide a two-player game representation for such control problems assuming time dependent dynamics and Lagrangian and the set constraints merely compact. Using viability results recently investigated for state constrained problems in an infinite horizon setting, we extend some known results for the linear quadratic regulator problem to a class of control problems with nonlinear dynamics in the state and affine in the control.  Feedback laws are obtained under suitable controllability assumptions.
	\end{abstract}

	
	
	
	\section{Introduction}

	
	The literature dealing with optimality conditions for  finite or infinite horizon optimal control problems, without state constraints, is quite rich (\cite{bascocannfrank2019semisubRiem,cannfrank2018valuefunction,Carlson:1987:IHO:38189,seierstad1986optimal}, and the references therein). Recovering feedback laws in the presence of state constraints, on the other hand, is challenging for infinite horizon problems (\cite{bascofrankcann2017necessary,vinterpappas1982amaximumprinciple}): when constraints are imposed on the state, or when barrier functions are involved, then finite horizon techniques typically fail for infinite horizon contexts (\cite{bascofrankowska2018hjbe}).
	
	We consider the following infinite horizon control problem:
	\eee{
		\ttnn{minimize }  \int_{t}^{\infty} \scr L(s,\xi\rs,u\rs)\,ds
	}
	over all $(\xi\ccd,u\ccd)$ satisfying the dynamics and state constraints described by
	\sistemanoref{
		\xi'\rs=\nabla h(\xi\rs)^{-1}A\rs h(\xi\rs)+\nabla h(\xi\rs)^{-1}B\rs u\rs & s\in[t,\infty)\ttnn{ a.e.}\\
		\xi(t)=x\\
		\xi\ccd\subset \Omega,
	}
	where $\Omega\subset \bb R^n$ is compact, $(t,x)\in \rzero\times \Omega$ is the initial datum, $A\ccd\in \bb R^{n\times n}$ and $B\ccd\in \bb R^{n\times m}$ are given time-dependent matrices, $\scr L:\bb R\times \bb R^n\times \bb R^m\ra \bb R^+$ is the Lagrangian, and
	$h:\bb R^n\ra \bb R^n$ is a diffeomorphism. We focus on Lagrangians as marginal functions, i.e.,
	\equazioneref{Lagrangiana_sup_intro}{
		\scr L(s,\xi,u)=\sup_{\alpha\mageq 0}\,\{\ps{h(\xi)}{Q(s,\alpha)h(\xi)}+\ps{u}{Ru}\},
	}
	where $Q(s,\alpha)\in \bb R^{n\times n}$ and $ R\in \bb R^{m\times m}$ are given positive symmetric matrices for all $s,\,\alpha\mageq 0$. In the special case where $h\ccd$ is the identity, convex Lagrangians can be rewritten using duality arguments in the form \cref{Lagrangiana_sup_intro} and optimality conditions are investigated (\cite{dowercantonimcene2019gamerepbarr,rockafellar1974conjugate,rockafellar2008linear}). Further specialization to the case where $\scr L$ is also quadratic in the state and control, yields a linear-quadratic regular (LQR) problem. LQR problems ar well studied in the context of convex control problems and Hamilton-Jacobi-Bellman developments for finite and infinite dimensional systems are well known (\cite{andersonmoore1971linearoptimalcontrol,DaPrato:2006:RepresentationControl,curtainprichard1978infinitelinear}). Solutions of a relevant parametrised family of finite horizon LQR problems, with running costs $\int_t^T( \ps{\xi\rs}{ Q^\alpha(s)\xi\rs}+\ps{u\rs}{R\rs u\rs})\,ds$ where ${Q}^\alpha\ccd=Q(\cdot,\alpha\ccd)$ and $\alpha:[t,T]\ra\bb R^+$ is a continuous function, are strictly related with the solutions of the set of Riccati ordinary differential equations
	\equazioneref{Riccati_intro}{
		P'+A^\star P+P A- P BR^{-1} B^\star P+ Q^\alpha=0\quad \ttnn{a.e.}
	}
	with final condition $P(T)=0$ (\cite{andersonmoore1971linearoptimalcontrol,DaPrato:2006:RepresentationControl}). Convex duality tools to study the LQR problem in the language of calculus of variations for finite time horizons problems have been developed and applied by Rockafellar (\cite{rockafellar1970conjugate,rockafellar1974conjugate,rockafellar2008linear,rockafellar2009variational}).  Moreover, Da Prato and Ichikawa (\cite{dapratoichikawa1987optimalcontrol,dapratoichikawa1988quadraticcontrollinearsystems}) studied, for almost-periodic dynamics, the solutions of the corresponding Riccati equations. However, when the system is subject to state constraints, or when non-quadratic costs or barrier functions are involved, the linear and quadratic techniques are no longer applicable. Recent work \cite{dowercantonimcene2019gamerepbarr} investigates, using convex duality techniques, two-player game representation results for LQR problems with convex state constraints imposed via barrier functions type (\cite{dowercantoni2017state,DMC2:16}). 
	
	In this work, we address the above state constrained control problem where the Lagrangian can be more generally expressed as in \cref{Lagrangiana_sup_intro}. The constraint set is assumed merely compact and no smoothness conditions are imposed on its boundary. We show that the associated value function can be written as a supremum of a parametrized set of value functions of quadratic control problems. Techniques from non-smooth analysis and viability theory are used to obtain the optimal synthesis for each parametrized problem. Furthermore, we provide controllability conditions to derive feedback laws in terms of a solution $P$ of the Riccati differential equation \cref{Riccati_intro} on the infinite horizon (\Cref{optimal_synthesis}). Such $P$ in general time dependent. However, when the dynamics and Lagrangian are time invariant.

	The outline of the paper is as follows. In \Cref{preliminaries} we provide basic definitions and facts from nonsmooth analysis and viability. \Cref{value_function} is devoted to the two-player game formulation of a large class of infinite horizon control problems with state constraints. In \Cref{optimal_synthesis}, we provide sufficient conditions for obtaining feedback laws of infinite horizon quadratic control problems, under state constraints and controllability assumptions.


	\section{Preliminaries}\label{preliminaries}
	
	We denote the set of nonnegative real numbers by $\rzero$ and the set of natural numbers by $\bb N$. $B(x,\delta)$ denotes the closed ball in $\mathbb{R}^k$ with radius $\delta>0$ centered at $x\in \mathbb{R}^k$ and $\mathbb{B}\doteq B(0,1)$. $|\cdot|$ and $\langle \cdot , \cdot \rangle$ denote the Euclidean norm and scalar product, respectively. With $C\subset \bb R^k$, the {interior} of $C$ is denoted by ${\rm int}\,C$, the closure of $C$ by $\overline{C}$,
	the boundary of $C$ by $\partial C$, and the distance from $x\in \mathbb{R}^k$ to $C$ by $d_C(x)\doteq \inf\{|x-y|\,:\,y\in C\}$. The negative polar cone of $C$, written $C^-$, is the set $\graffe{v\in \bb R^k\,:\,\ps{v}{x}\mineq 0\; \forall x\in C}$. The set of all $n\times m$ real matrices $M$ is denoted by $\bb R^{n\times m}$, endowed with the norm $\norm{M}=\sup_{x\neq 0} {\modulo{Mx}}/{\modulo{x}}$. If $M\in \bb R^{n\times m}$, $M^\star$ stands for the transpose matrix of $M$ and, if $M$ is invertible, we write $M^{-\star} \doteq (M^\star)^{-1}$. A matrix $M\in \bb R^{n\times n}$ is said to be $r$-negative definite if $r>0$ and $\ps{Mx}{x}\mineq -r|x|^2$ for all $x\in \bb R^n$.
	
	For $p\in \mathbb{R}^+\cup \{\infty\}$ and a Lebesgue measurable set $I\subset \bb R^n$ we denote by $L^p(I;\mathbb{R}^k)$ the space of $\mathbb{R}^k$-valued Lebesgue measurable functions on $I$ endowed with the norm $\|\cdot\|_{p,I}$ (if $I=[a,b)$ we write $L^p(a,b;\mathbb{R}^k)$). We say that $f\in L^p_{{\rm loc}}(I;\mathbb{R}^k)$ if $f\in L^p(J;\mathbb{R}^k)$ for any compact subset $J\subset I$. We denote by $C(I;\bb R^k)$ the space of all continuous $\bb R^k$-valued functions. If $I$ is open, then we denote by  $C^1(I;\bb R^k)$ the space of all continuously differentiable $\bb R^k$-valued functions on $I$.
	The set of all measurable functions $\alpha:\rzero\ra \rzero$ is denoted by $\ccal A$.
	
	Let \(Y\) be a measurable space and $X$ a Banach space. Consider a set-valued map \(F : Y \rightsquigarrow X\). We denote by $\ttnn{dom }F$ the \textit{domain} of $F$, i.e., the set of all $y\in Y$ such that $F(y)\neq \emptyset$. A measurable function \(f : Y \ra X\) satisfying \(f(y) \in F(y)\) for all $y \in Y$
	is called a \textit{measurable selection} of \(F .\) Existence of a measurable selection may be guaranteed via \cite[Theorem 8.1.3]{aubin2009set}.
	\begin{proposition}[\cite{aubin2009set}]\label{measurableselection}
		Let \(F : Y \rightsquigarrow X\) be a measurable set-valued map with closed nonempty values. Then there exists a measurable selection of \(F\).
	\end{proposition}

	Let $D\subset \bb R^n$ be nonempty and $\graffe{A_y}_{y\in D}$ be a family of nonempty subsets of $\bb R^k$. The Kuratowski-Painlev\'e \textit{upper and lower limits} (\cite{rockafellar2009variational}) of $A_h$ at $x\in D$ are the sets defined, respectively, by
	\eee{
		\Limsup_{y\ra x,\, y\in D} A_y=\{v\in \bb R^k\,:\, \liminf_{y\ra x,\,y\in D} d_{A_y}(v)=0\},\;\;\;	\Liminf_{y\ra x,\, y\in D} A_y=\{v\in \bb R^k\,:\, \limsup_{y\ra x,\, y\in D} d_{A_y}(v)=0\}.
	}
	
	
	Assume now $X=Y=\bb R^k$. $F$ is said to be \textit{upper semicontinuous at} $x\in \bb R^k$ if $x\in \ttnn{dom }F$ and for any $\varepsilon>0$ there exists $\delta>0$ such that $F(y)\subset F(x)+\varepsilon \bb B$ for all $\modulo{y-x}\mineq \delta$. If $F$ is upper semicontinuous at every $x\in \bb R^k$ then $F$ is said to be \textit{upper semicontinuous}.
	$F$ is said to be \textit{lower semicontinuous at} $x\in \bb R^k$ if $\Liminf_{y\ra x}F(y)\subset
	F(x)$. $F$ is said to be \textit{lower semicontinuous} if $F$ is
	lower semicontinuous at every $x\in \bb R^k$.  $F$ is called
	\textit{continuous at} $x\in \bb R^k$ if it is lower and upper
	semicontinuous at $x$ and  it is continuous if it is continuous at
	each point $x$.
	The set valued map $F$ is said to be $k$-\textit{Lipschitz continuous}, for some $k\mageq 0$, if $F(x)\subset F(\tilde x)+ k\modulo{x-\tilde x}\bb B$ for all $x,\,\tilde x\in \bb R^k$.
	
	Consider a nonempty subset $E\subset \bb R^k$ and $x\in \overline{E}$. The \textit{contingent cone} $T_E(x)$ to $E$ at $x$ is defined as the set of all vectors $v\in \bb R^k$ such that $\liminf_{h\ra 0+} {d_E(x+hv)}/h=0$.
	The \textit{limiting normal cone} to $E$ at $x$, written $N_E(x)$, is defined with respect to the negative polar cone of the contingent cone by $N_E(x)\doteq \Limsup_{y\ra x,\, y\in E}{T_E(y)}^-$. It is known that  $N_E(x)^-\subset T_E(x)$ and 
	\eee{
		x\rightsquigarrow N^1_E(x)\doteq N_E(x)\cap \partial \bb B
	}
	is upper semicontinuous (\cite{Aubin:1991:viabilitytheory}), whenever $E$ is closed. The following viability result is a particular case of a more general one (\cite[Theorem 4.2]{frankowskaplaskrze1995measviabth}).
	
	\begin{proposition}[\cite{frankowskaplaskrze1995measviabth}]\label{viability}
		Let $0\mineq t<T$ and $F:[t,T]\times \bb R^k\rightsquigarrow \bb R^k$ be a measurable set-valued map with closed convex values such that $F(s,\cdot)$ is continuous for a.e. $s\in [t,T]$ and there exists $\theta\in L^1(t,T;\rzero)$ satisfying $\sup_{v\in F(s,x)}|v|\mineq \theta(s)$ for all $x\in \bb R^k$ and for a.e. $s\in [t,T]$. Consider a closed subset \(K \subset \bb R^k\) and suppose that
		\eee{
			F(s,x)\cap T_K(x)\neq \emptyset\quad \ttnn{for a.e. }s\in [t,T],\, \forall x\in \partial K.
		}	
		Then for all \(x_{0} \in K\) there exists a solution $\xi\ccd$ on $[t,T]$ to the differential inclusion $\xi'\rs\in F(s,\xi\rs)$ such that $\xi\ccd\subset K$ and $\xi(t)=x_0$.	
	\end{proposition}
	
	\begin{remark}\label{remark_extension_set_valued_map}\rm
		To apply \cref{viability} to locally bounded set-valued maps we extend it in the following way: let $\tilde G:[t,T]\times  \bb R^k\rightsquigarrow \bb R^{k}$ be a set-valued map  such that for all $R>0$ there exists $\tilde \theta_R\in L^1(t,T;\rpiu)$ satisfying  $\sup_{v\in \tilde G(s,x)}|v|\mineq \tilde \theta_R(s)$ for a.e. $s\in [t,T]$ and all $x\in B(0,R)$. Now consider the set-valued map $G_*:[t,T]\times \bb R^{k}\rightsquigarrow \bb R^{k}$ defined by $G_*(s,x)=\tilde G(s,x)$ for any $(s,x)\in [t,T]\times B(0,M)$ and $G_*(s,x)=\tilde G(s,\pi(x))$ for any $(s,x)\in [t,T]\times (\bb R^{k}\backslash B(0,M))$, where $\pi\ccd$ stands for the projection operator onto $B(0,M)$, i.e., $|\pi(x)-x|=d_{ B(0,M)}(x)$, $M\doteq R+{\int_{t}^{T}\tilde \theta_R(s)\,ds}$, $R\doteq\max_{x\in \Omega}|x|$, and $\Omega$ as in the Introduction. Then, for a suitable $R'>R$, we have $	\sup_{v\in G_*(s,x),\, x\in \bb R^{k}}\modulo{v}\mineq \tilde \theta_{R'}\rs$ for a.e. $s\in [t,T]$. Thus, $\xi:[t,T]\ra \bb R^{k}$, with $\xi(t)\in B(0,R)$, satisfies $\xi'\rs\in G_*(s,\xi\rs)$ if and only if $\xi'\rs\in \tilde G(s,\xi\rs)$.
	\end{remark}
	\section{Value function and two-player game representation}\label{value_function}
	
	Consider the problem
	\equazioneref{funzionalecosto}{
		{\rm  minimize } \; J(t,x,u)\doteq \int_{t}^{\infty}\scr L(s,\xi(s) ,u (s) )\,ds
	}
	over all trajectory-control pairs $(\xi(\cdot),u\ccd)$ satisfying the state constrained system
	\sistemaref{sistemacontrollo}{
		\xi' (s) =f_0(s,\xi(s))+f_1(s,\xi\rs)u\rs & s\in[t,\infty)\ttnn{ a.e.}  \\
		\xi(t)=x\\
		\xi(\cdot)\subset \Omega,
	}
	where
	$(t,x)\in \rzero\times \Omega$ is the initial datum, and the dynamics $f_0:\bb R\times \bb R^n\ra \bb R^n$, $f_1:\bb R\times \bb R^n \ra \bb R^{n\times m}$ and the Lagrangian $\scr L:\bb R\times \bb R^n\times \bb R^m\ra \bb R^+$ are functions measurable in time. In the following we set for any $(s,\xi,u)\in \bb R\times \bb R^n\times \bb R^m$
	\eee{
		f(s,\xi,u)\doteq f_0(s,\xi)+f_1(s,\xi)u.	}

	\begin{definition}\label{def_feasible}
		A trajectory-control pair $(\xi(\cdot),u\ccd)$ that satisfies the state constrained system \cref{sistemacontrollo} is called \textit{feasible} (we also refer to such a trajectory as feasible). The set of all controls such that the associated trajectory is feasible at the initial datum $(t,x)$ is denoted by $\ccal U{(t,x)}$. For any $u\in \ccal U(t,x)$ we denote by $\xi_{u}\ccd$ the trajectory solving \cref{sistemacontrollo} associated with the control $u\ccd$ and starting from $x$ at time $t$.
	\end{definition}

	The function $W:\rzero\times \Omega \to  \mathbb{R}\cup \{\pm\infty\}$
	\equazioneref{problema_W_inf}{
		W(t,x)=\inf_{u\in \ccal U(t,x)}J(t,x,u)
	}
	is called the {\em value function} of problem \cref{funzionalecosto}-\cref{sistemacontrollo}. By convention $W(t,x)\doteq +\infty$ if no feasible trajectory-control pair exists at $(t,x)$ or if the integral in \cref{funzionalecosto} is not defined for every feasible pair. A control $u\in \ccal U(t,x)$ is said to be \textit{optimal} at $(t,x)$ if $W(t,x)=\int_{t}^{\infty}\scr L(s,\xi_u(s) ,u (s) )\,ds$. Recall that for a function $q\in L^1_{{\rm loc}}(t,\infty;\mathbb{R})$ the aforementioned integral $\int_{t}^\infty q (s)$ s defined by $\lim_{T \to \infty}\int_{t}^T q (s) \,ds$, provided this limit exists.
	
	We consider the following assumptions on $f_0,\,f_1,$ and $\scr L$:
	\begin{hypothesis}\label{ipotesi_H}[\cref{ipotesi_H}]
		\
		
		\begin{enumerate}[label=(\roman*)]
			
			\item\label{H_convessita_f_L_epigrafico} the set $\{(f_1(s,x)u,\scr L(s,x,u)+r)\,:\,u\in \bb R^m, r\mageq 0 \}$ is closed and convex for all $s\in \rzero$, $x\in\mathbb{R}^{n}$;
			
			\item\label{H_Lip_f_L} there exists $ k\in L^1_{\ttnn{loc}}(\rpiu ;\mathbb{R}^+)$ such that $f_0(s,\cdot)$, $f_1(s,\cdot)$, and $\scr L(s,\cdot,u)$ are $k\rs$-Lipschitz continuous for a.e. $s\in \rzero$ and uniformly for all $u\in \bb R^m $;
			
			\item\label{H_integrabilita_L2_f_L_nel_tempo} given any $r>0$, there exists $\theta_r\in L^2_{\ttnn{loc}}(\rpiu;\rpiu)$ such that $|f_0(s,x)|+\norm{f_1(s,x)}\mineq \theta_r(s)$ for a.e. $s\in \rzero$ and all $x\in B(0,r)$;
			
			\item\label{H_L_mageq_u_piu_psi} there exists a function $\phi\in L^1_{\ttnn{loc}}(\rpiu;\rpiu)$ such that $\scr L(s,x,u)\mageq |u|^2-\phi(s)$ for a.e. $s\in \rzero$ and for all $x\in \mathbb{R}^n$, $u\in \bb R^m $.
			
		\end{enumerate}
	\end{hypothesis}
	
	%

	\begin{proposition}\label{proplscW}
		Assume \cref{ipotesi_H}. Then for all $(t,x)\in \ttnn{dom } W$ there exists an optimal control for $W$ at $(t,x)$ and $W$ is lower semicontinuous.
	\end{proposition}
	\begin{proof}
		Let $ \left(t , x  \right) \in \operatorname { dom } W $, $\graffe{u_j}_{j\in \bb N}\subset \ccal U(t,x)$ a minimizing sequence for $W(t,x)$, and denote by $\xi_j$ the trajectory starting from $x$ and associated to the control $u_j$. We notice that, since $\Omega$ is compact, the set $\graffe{\xi_{ j }}_{j\in \bb N}$ is equibounded. Moreover, by \cref{ipotesi_H}-\ref{H_L_mageq_u_piu_psi}, for any $T>t$ we have $\norm{u_j}_{2,(t,T)}\mineq \int_t^T \scr L(s,\xi_j(s),u_j\rs)\,ds+\int_t^T \phi(s)\,ds$ for all $j\in \bb N$. So, $\{\norm{u_j}_{2,(t,T)}\}_{j\in \bb N}$ is bounded by a constant $C_T>0$. Then, putting $R=\max_{x\in \Omega}|x|$, by \cref{ipotesi_H}-\ref{H_integrabilita_L2_f_L_nel_tempo} and H\"{o}lder we get for all $T>t$, $t\mineq \tilde \tau<  \tau\mineq  T$, and all $j\in \bb N$,
		\equazioneref{equintegrabilita_equicontinuita}{
			|\xi_j(\tau)-\xi_j(\tilde \tau)|&\mineq \int_{\tilde \tau}^\tau|f(s,\xi_j\rs,u_j\rs)|\,ds\\
			&\mineq \int_{\tilde \tau}^\tau|f_0(s,\xi_j\rs)|\,ds +\int_{\tilde \tau}^\tau \norm{f_1(s,\xi_j\rs)}|u_j\rs|\,ds\\
			&\mineq \sqrt{\tau-\tilde \tau}\norm{\theta_R}_{2,(\tilde \tau,\tau)}+C_T\norm{\theta_R}_{2,(\tilde \tau,\tau)},
		}	
		so that $\{ \xi_j \}_{j\in \bb N}$ is equicontinuous.
		From Ascoli-Arzel\`{a}'s theorem and the closedness of $\Omega$, there exists a subsequence $\graffe{\xi _ { j_ k }}_{k\in \bb N}$ converging almost uniformly to a continuous function $\overline { \xi } : [t,\infty)\rightarrow  \bb R^n$ satisfying $\bar \xi\ccd\subset \Omega$. From \cref{equintegrabilita_equicontinuita} and applying the Dunford-Pettis theorem (\cite{brezis2010functional}), taking a subsequence and keeping the same notation, we have for some $y \in L^1_{\ttnn{loc}}(t,\infty;\bb R^n)$ and $\beta \in L^1_{\ttnn{loc}}(t,\infty;\rzero)$ that for all $T> t$: $ \xi _ { j_k } ^ { \prime } \rightharpoonup y$ in $L ^ { 1 } \left( t , T; \mathbb { R } ^ { n } \right)$ and $ \scr L(\cdot,\xi _ { j_k } ( \cdot ) , u _ { j_k }\ccd )\rightharpoonup \beta$ in $L ^ { 1 } \left( t , T; \rzero\right )$. Passing to the limit yields $\overline { \xi } ( s) = x + \int _ { t} ^ { s} y ( \tau) d \tau $ for all $s\mageq t$. So, $ {\bar \xi }$ is locally absolutely continuous and, applying the Lebesgue theorem, $\overline { \xi } ^ { \prime } ( s) = y ( s )$ for a.e. $s\in [t,T]$. Moreover, since $\scr L\mageq 0$, for any $T> t $ we have $	\int _ { t } ^ { \infty } \scr L \left( s , \xi _ { j_k } ( s) , u _ { j_k } ( s) \right) d s \geq \int _ { t } ^ { T} \scr L \left( s , \xi _ { j_k } ( s) , u _ { j_k } ( s ) \right) d s$ for all $k\in \bb N$. Taking the limit as $k\ra \infty$, it follows that $W \left( t , x \right) \geq \int _ { t  } ^ { T} \beta ( s) d s$. By arbitrariness of $T>t$, we deduce $W \left( t , x  \right) \geq \int _ { t } ^ { \infty } \beta ( s ) d s$. Now, we show that there exist a measurable control $\bar u\ccd$ and a measurable function $r :[t,\infty)\rightarrow \bb { R } ^{ + }$ such that $\bar \xi\ccd$ and $\beta\ccd$ satisfy
		\equazioneref{stima1_proplscW}{
			\overline { \xi } ^ { \prime } ( s ) = f ( s, \overline { \xi } ( s ) , \overline { u } ( s ) ),\quad \beta ( s ) =\scr L ( s , \overline { \xi } ( s ) , \overline { u } ( s ) ) + r ( s),
		}
		for a.e. $s\mageq t.$
		Denote by $G:\bb R\times \bb R^n\rightsquigarrow \bb R^n\times \bb R$ the set-valued map defined by 
		\eee{
			G(s,x)=\{(f(s,x,u),\scr L(s,x,u)+r)\,:\,u\in \bb R^m, r\mageq 0 \}.
		}
		From assumption \cref{ipotesi_H}-\ref{H_Lip_f_L}, we can assume that for any $T>t$ there exists $q\in L^1(t,T;\rzero)$ such that for a.e. $s\in[t,T]$
		\eee{
			\left( \xi _ { j_k } ^ { \prime } ( s ) , \scr L \left( s, \xi _ { j_k } ( s ) , u _ { j_k } ( s) \right) \right) \in G \left( s , \xi _ { j_k } ( s ) \right)\subset G ( s , \overline { \xi } ( s ) ) +  q ( s )| \xi _ { j_k } ( s ) - \overline { \xi} ( s)| \bb  B.
		} 
		Let $\varepsilon > 0$, then there exists $k_\varepsilon\in \bb N$ such that $( \xi _ { j_k } ^ { \prime } ( s) , \scr L \left( s , \xi _ { j_k } ( s ) , u _ { j_k } ( s) \right) ) \in G ( s , \overline { \xi } ( s ) ) +  q ( s ) \varepsilon\bb B$ for a.e. $s\in[t,T]$ and all $k\mageq k_\varepsilon$. We notice that, by \cref{ipotesi_H}-\ref{H_convessita_f_L_epigrafico}, $G ( s, \overline { \xi } ( s) ) +  q( s) \varepsilon  \bb  B$ is closed and convex. Hence, applying Mazur's theorem (\cite{brezis2010functional}), we deduce that $( \overline { \xi } ^ { \prime } ( s ) , \beta ( s) ) \in G ( s , \overline { \xi } ( s ) ) +  q  ( s ) \varepsilon\bb B$ for a.e. $s\in[t,T]$. Since $\varepsilon$ is arbitrary, $( \overline { \xi } ^ { \prime } ( s ) , \beta ( s ) ) \in G (s , \overline { \xi } ( s ) )$ for a.e. $s\in[t,T]$ and therefore $( \overline { \xi } ^ { \prime } ( s ) , \beta ( s ) ) \in G ( s , \overline { \xi } ( s ) )$	for a.e. $s\mageq t$. Now, from the measurable selection theorem, there exist a control $\overline { u } ( \cdot )$ and a measurable function $r :[t,\infty)\rightarrow \bb { R } ^{ + }$ satisfy-ing \cref{stima1_proplscW}. Notice that $\bar u\in \ccal U(t,x)$. Thus, from \cref{stima1_proplscW}, $W\left( t , x  \right) \geq \int _ { t } ^ { \infty } \scr L ( s , \overline { \xi} ( s ) , \overline { u } ( s ) ) d s$, and, finally, $( \overline { \xi } , \overline { u } )$ is optimal at $(t,x)$.
		
		Now, we prove the lower semicontinuity of $W$. Consider $\graffe{(t_j,x_j)}_{j\in \bb N}$ converging to $(t,x)$ in $\ttnn{dom }W$ and denote by $u_j\in \ccal U(t_j,x_j)$ the minimizers. Keeping the same notation as above, we may conclude that there exists a subsequence $\graffe{\xi _ { j_ k }}_{k\in \bb N}$ converging almost uniformly to an absolutely continuous function $\overline { \xi } : [t,\infty)\rightarrow \Omega$ such that $\xi_{j_k}(t_{j_k})=x_{j_k}\ra \bar \xi(t)=x$ as $k\ra \infty$ and $\beta \in L^1_{\ttnn{loc}}(t,\infty;\rzero)$ satisfying $	\liminf_k W ( t_{j_k} ,x_{j_k}) \geq \int _ { t } ^ { \infty } \beta ( s ) d s$. Then the lower semicontinuity follows arguing as in the first part and the proof is complete.
	\end{proof}

	\begin{proposition}\label{viab_tramite_ipc}
		Assume \cref{ipotesi_H} and
		\equazioneref{ipc_base}{
			\{f(s,y,u)\,:\,u\in \bb R^m\}\cap \ttnn{int } T_\Omega(y)\neq \emptyset
		}
		for a.e. $s \in \rzero$ and all $y\in  \partial \Omega$. Then $\ccal U(t,x)\neq \emptyset$ for any $(t,x)\in  \rzero \times \Omega$.
	\end{proposition}
	\begin{proof}
		Notice that, from \cref{ipc_base}, for a.e. $s\mageq 0$ and all $y\in \bb R^n$, there exists $\delta_{s,y}>0$ such that $\{f(s,y,u)\,:$ $|u|\mineq \delta_{s,y}\}\cap \ttnn{int }T_\Omega(y)\neq \emptyset$. From assumption \cref{ipotesi_H}-\ref{H_Lip_f_L} and the compactness of $\Omega$, the set-valued map $y\rightsquigarrow N_\Omega^1(y)$ is upper semicontinuous.  Using a compactness argument, we can find $\delta>0$ such that $F(s,y)\,\cap\, T_\Omega(y)\neq \emptyset$ for a.e. $s\mageq 0$ and all $y\in \partial \Omega$, where we defined $F(s,y)\doteq \{f(s,y,u)\,:\,|u|\mineq \delta\}$. Now, fix $(t,x)\in \rzero\times \Omega$. From \cref{ipotesi_H}-\ref{H_Lip_f_L} and \cref{remark_extension_set_valued_map}, applying \cref{viability} and the measurable selection theorem  on the time interval $[t,t+1]$ to the set-valued map $F$, there exist $u^0\ccd$ and $\xi^0\ccd$ feasible solving \cref{sistemacontrollo} on $[t,t+1]$ with $\xi^0(t)=x$. Using again \cref{viability} and \cref{remark_extension_set_valued_map} on the time interval $[t+1,t+2]$, keeping as initial state $\xi^0(t+1)\in \Omega$, we have that there exist a control $u^1\ccd$ and $\xi^1\ccd$ feasible solving \cref{sistemacontrollo} on $[t+1,t+2]$ and starting from $\xi^0(t+1)$. So, we may conclude that for all $j$ there exist $u^j\ccd$ and $\xi^j\ccd$ solving \cref{sistemacontrollo} on $[t+j,t+j+1]$ and $\xi^{j+1}(t+j+1)=\xi^j(t+j+1)$. Thus, the conclusion follows now considering the feasible trajectory, starting from $x$ at time $t$, defined by $\xi(s)\doteq \xi^j(s)$ if $s\in [t+j,t+j+1]$.
	\end{proof}
	
	\begin{remark}\label{remark1}
		\Cref{viab_tramite_ipc} ensures the existence of feasible trajectories under the condition \cref{ipc_base}, which is referred to as an \textit{inward pointing condition} (i.p.c.). The (i.p.c.) has been extended to less restrictive frameworks (\cite{bascofrankowska2018lipschitz,Soner86a}). Such an assumption requires, roughly speaking, that at each point on the boundary of the constraint set $\partial\Omega$ there exists an admissible velocity pointing into its interior. Furthermore, \cref{ipotesi_H}-\ref{H_convessita_f_L_epigrafico} cannot be weakened by assuming the convexity of the set $\{(f(s,x,u),$ $\scr L(s,x,u))\,:\,u\in \bb R^m \}$ since, in many applications, the Lagrangian is not affine in the control.
	\end{remark}
	
	In the following we assume that $\scr L$ is a marginal function, i.e.,
	\eee{
		\scr L(s,\xi,u)=\sup_{\alpha\mageq 0}\,\ell  (s,\xi,u,\alpha),
	}
	where $\ell  (s,\xi,u,\alpha)=\ell_1 (s,\xi,\alpha)+\ell_0(s,u)$ with $\ell_0 :\bb R\times  \bb R^m \ra \bb R,\,\ell_1 :\bb R\times \bb R^n\times \bb R \ra \bb R$ fun-ctions measurable in time. For any $\alpha\in \ccal A$ we define the value function $W^\alpha:\rzero\times \Omega\ra \bb R\cup \graffe{\pm \infty}$ of the auxiliary control problem
	\equazioneref{W_alpha}{
		W^\alpha(t,x)\doteq \inf_{u\in \ccal U(t,x)}J_\alpha(t,x,u)
	}
	where $J_\alpha(t,x,u)\doteq  \int_{t}^{\infty}\ell  (s,\xi_u(s) ,u (s),\alpha(s) )\,ds$ and $\ccal U(t,x)$ is as in \Cref{def_feasible}.

	\begin{hypothesis}\label{ipotesi_H_primo}[\cref{ipotesi_H_primo}] 
		\
		
		\begin{enumerate}[label=(\roman*)]
			
			\item\label{H_primo_i} \cref{ipotesi_H} holds with $k\in L^1(\rpiu;\rpiu)$ and $\phi\in L^2(\rpiu;\rpiu)$;
			
			\item\label{H_primo_lip_l_1} there exists $k_1\in L^2(\rpiu;\rpiu)$ such that $|\ell_1(s,x,\alpha)-\ell_1(s,x,\hat \alpha)|\mineq k_1(s)|\alpha-\hat \alpha|$ for a.e. $s\in \rzero$ and all $x\in \Omega$, $\alpha,\hat \alpha\in \rzero$;
			
			\item\label{H_primo_Lambda_Lip} there exists $\psi\in L^2(\rpiu;\rpiu)$ such that
			\equazioneref{set_valued_map_Lambda}{
				x\rightsquigarrow\Lambda(s,x)\doteq \{\alpha\mageq 0\,:\,\sup_{\beta\mageq 0} \ell_1(s,x,\beta)= \ell_1 (s,x,\alpha)\}
			}
			is $\psi(s)$-Lipschitz for all $s\in \rzero$;
			
			\item\label{H_primo_Frechet_diff_J} $J(s,x,\cdot)$ and $J_\alpha(s,x,\cdot)$ are Fr\'{e}chet differentiable on $L^2(s,\infty;\bb R^m)$ for all $\alpha\in \ccal A$, $s\in \rzero$, and $x\in \Omega$.
		\end{enumerate}
	\end{hypothesis}

	\begin{lemma}\label{Lemmachiave}
		Assume \cref{ipotesi_H_primo}-\ttnn{\ref{H_primo_i}-\ref{H_primo_lip_l_1}} and that $\Lambda(s,\cdot)$ takes closed nonempty values for all $s\in \rzero$. Then for all $(t,x)\in \rzero\times\bb R^n$:
		\enurom{
			\item $W(t,x) = \inf\{J(t,x,u)\,:\,u\in \ccal U(t,x)\cap L^2(t,\infty;\bb R^m)\} $;
			
			\item for any $u \in \ccal U(t,x)$,
			\equazioneref{tesilemma1}{
				J(t,x,u)&=\sup_{\alpha\in \ccal A}\int_{t}^{\infty}\ell  (s,\xi_u(s) ,u (s),\alpha(s) )\,ds;
			}
			
			\item\label{lemma_statement_alpha_delta_meas_selection} if $u \in \ccal U(t,x)$ and $\alpha^u\ccd\in \Lambda(\cdot,\xi_u\ccd)$ is a Lebesgue measurable selection on $[t,\infty)$, then the supremum in \cref{tesilemma1} is attained for $\alpha^u$ whenever $J(t,x,u)<\infty$.
		}
		
	\end{lemma}
	\begin{proof}
		Fix $(t,x)\in \rzero\times\bb R^n$.
		
		The statement $(i)$ follows from the assumption \cref{ipotesi_H}-\ref{H_L_mageq_u_piu_psi}.
		
		Next we prove $(ii)$. Let $u\in \ccal U(t,x)$. We claim that
		\equazioneref{claim1}{
			\int_{t}^{\infty}\scr L(s,\xi_u(s) ,u (s) )\,ds\mageq \sup_{\alpha\in \ccal A}\int_{t}^{\infty}\ell  (s,\xi_u(s) ,u (s),\alpha(s) )\,ds.
		}
		If $\int_{t}^{\infty}\scr L(s,\xi_u(s) ,u (s) )\,ds=+\infty$, then the claim follows. Otherwise, since for all $s\mageq 0, \xi\in \bb R^n, u\in \bb R^m,$ and $\alpha\mageq 0$ we have that $\scr L(s,\xi,u)\mageq \ell  (s,\xi,u,\alpha)$. Hence, for all $u \in \ccal U(t,x)$ and all  $\alpha\in \ccal A$ we have $ \int_{t}^{\infty}\scr L(s,\xi(s) ,u (s) )\,ds$ $\mageq \int_{t}^{\infty}\ell  (s,\xi(s) ,u (s),\alpha(s) )\,ds$, and the claim follows. We show next the inverse inequality in \cref{claim1}. Assume that the right-hand side of \cref{claim1} is finite. Fix $u\in \ccal U(t,x)$. Then for any $\alpha\in \ccal A$ the function $s\mapsto  \ell(s,\xi_u(s) ,u (s),\alpha(s) )$ is locally integrable on $[t,\infty)$. Since $s\rightsquigarrow \Lambda(s,\xi_u(s))$ has closed nonempty values, applying the measurable selection theorem, there exists a measurable function $\tilde \alpha:[t,\infty)\ra \rzero$ satisfying $\scr L(s,\xi_u(s) ,u (s) )=\ell  (s,\xi_u(s) ,u (s),\tilde \alpha(s) )$ for a.e. $s\mageq t$.
		Hence, since $\scr L\mageq 0$, for all $T>t$
		\equazioneref{lemma1_ultima}{
			\int_{t}^{T}\scr L(s,\xi_u(s) ,u (s) )\,ds&= \int_{t}^{T}\ell  (s,\xi_u(s) ,u (s),\tilde \alpha(s) )\,ds\\
			&\mineq \int_{t}^{\infty}\ell  (s,\xi_u(s) ,u (s),\tilde \alpha(s) )\,ds\\
			&\mineq \sup_{ \alpha\in \ccal A}\int_{t}^{\infty}\ell  (s,\xi_u(s) ,u (s), \alpha(s) )\,ds.
		}
		Thus, the proof of \cref{tesilemma1} is complete passing to the limit as $T\ra \infty$ in \cref{lemma1_ultima}.
		
		The last statement (iii) follows immediately from \cref{claim1} and \cref{lemma1_ultima}.
	\end{proof}

	The next result provide a two-player game formulation for the value function of the control problem \cref{funzionalecosto}-\cref{sistemacontrollo}.
	
	\begin{proposition}\label{prop_W_uguale_sup_W_alpha}
		Assume \cref{ipotesi_H_primo}. Let $\bar u\ccd$ be optimal at $(t,x)\in \ttnn{dom}\,W\neq \emptyset$ and $\bar \alpha\ccd\in \Lambda(\cdot,\xi_{\bar u}\ccd)$ be a measurable selection on $[t,\infty)$. Suppose that $\xi_{\bar u}\ccd\subset \ttnn{int }\Omega$, there exist $C>0$ and $\bar \delta>0$ satisfying $\forall (s,\delta,w)\in(t,\infty)\times (0,\bar \delta)\times B_{L^2(t,\infty;\bb R^m)}(0,1)$
	\equazioneref{tubo}{
	|\xi_{\bar u+\delta w}\rs-\xi_{\bar u}\rs|\mineq C\delta
	}
	with $\xi_{\bar u+\delta w}(t)=x$, and $J_{\bar \alpha}(t,x,\cdot)$ is strictly convex in a neighborhood of $\bar u$.  Then
		\equazioneref{tesi_theorem_prop_W_uguale_sup_W_alpha}{
			W(t,x)=\sup_{\alpha\in \ccal A} W^\alpha(t,x).
		}
		
	\end{proposition}
	\begin{proof} Let $w\in \ccal U(t,x)$ and $\alpha\in \ccal A$. Applying \cref{Lemmachiave} we get
		\eee{
			\int_{t}^{\infty}\scr L(s,\xi_w(s) ,w (s) )\,ds&\mageq  \int_{t}^{\infty}\ell  (s,\xi_w(s) ,w (s),\alpha(s) )\,ds \\
			&\mageq \inf_{u\in \ccal U(t,x)} \int_{t}^{\infty}\ell  (s,\xi_u(s) ,u (s),\alpha(s) )\,ds,
		}	
		so $W(t,x)\mageq \sup_{\alpha\in \ccal A} W^\alpha(t,x)$.
		
		On the other hand, assume $\sup_{\alpha\in \ccal A} W^\alpha(t,x)<+\infty$. Let $ \bar \alpha\ccd\in \Lambda(\cdot,\xi_{\bar u}\ccd)$ be a mea-surable selection on $[t,\infty)$. From \cref{Lemmachiave}-\ref{lemma_statement_alpha_delta_meas_selection} it follows that \cref{tesilemma1} is satisfied along the pair $(\bar u\ccd,\bar \alpha\ccd)$. So, it is sufficient to show that
		\equazioneref{claim2}{
			\int_{t}^{\infty}\scr L(s,\xi_{\bar u}(s) ,\bar u (s) )\,ds\mineq \int_{t}^{\infty}\ell  (s,\xi_{w}(s) ,w (s),\bar \alpha(s) )ds,
		}
		for all $w\in \ccal U(t,x).$
		Fix $\eps>0$ and $w\in L^2(t,\infty;\bb R^m)$ with $\norm{w}_{2, (t,\infty)}=1$. From our assumptions, there exists $\delta_\eps\in (0,\eps)$ such that $J_{\bar \alpha}(t,x,\cdot)$ is Fr\'{e}chet differentiable and strictly convex on $\{u\in L^2(t,\infty;\bb R^m)\,:\,\norm{u-\bar u}_{2,(t,\infty)}\mineq \delta_\eps\}$.
		Since $\xi_{\bar u}\ccd\subset \ttnn{int }\Omega$ and from \rif{tubo}, replacing $\delta_\eps$ with a suitable small constant $\delta_\eps\in (0,\bar \delta)$, we have $\xi_{\bar u+\delta w}\ccd\subset \Omega$ for all $\delta\in (0,\delta_\eps)$. We may assume that $J(t,x,\bar u+\delta w)<\infty$ for any $\delta\in (0,\delta_\eps)$. In order to prove \cref{claim2}, it is sufficient to show that $D_u J_{\bar \alpha}(t,x,\bar u)( w)\mageq 0$, where $D_u$ stands for the Fr\'{e}chet derivative with respect to the variable $u$. For all $\delta\in (0,\delta_\eps)$ denote by $\alpha^\delta\ccd$ the measurable function satisfying the statement of \cref{Lemmachiave}-\ref{lemma_statement_alpha_delta_meas_selection}. We have
		\equazioneref{stimaparte1}{
			&J_{\bar \alpha}(t,x,\bar u+\delta w)-J_{\bar \alpha}(t,x,\bar u)\\
			&=\int_t^\infty (\ell_1(s,\xi_{\bar u+\delta w}\rs, \alpha^\delta\rs)+\ell_0(s,\xi_{\bar u+\delta w}\rs))\,ds\\
			&\qquad-{\int_t^\infty ( \ell_1(s,\xi_{\bar u}\rs, \bar \alpha\rs)+\ell_0(s,\xi_{\bar u}\rs))\,ds}\\
			&\qquad +\int_t^\infty (\ell_1(s,\xi_{\bar u+\delta w}\rs,\bar \alpha\rs) - \ell_1(s,\xi_{\bar u+\delta w}\rs, \alpha^\delta\rs))\,ds\\
			&=\int_t^\infty (\scr L(s,\xi_{\bar u+\delta w}\rs,\bar u\rs+\delta w\rs,\bar \alpha\rs) -\scr L(s,\xi_{\bar u}\rs,\bar u\rs,\bar \alpha\rs))ds\\
			&\qquad +\int_t^\infty (\ell_1(s,\xi_{\bar u+\delta w}\rs,\bar \alpha\rs)- \ell_1(s,\xi_{\bar u+\delta w}\rs, \alpha^\delta\rs))\,ds.
		}
		Now, from \cref{tubo} and assumption \cref{ipotesi_H_primo}-\ref{H_primo_Lambda_Lip}, it follows that there exists a small $\tilde \delta_\eps\in (0,\delta_\eps)$ such that $|\alpha^\delta\rs-\bar \alpha\rs|\mineq C\eps\psi(s)$ for a.e. $s\mageq t$ and all $\delta\in (0,\tilde \delta_\eps)$. So, by \cref{ipotesi_H_primo}-\ref{H_primo_lip_l_1}, we have that
		\eee{
			&\int_t^\infty|\ell_1(s,\xi_{\bar u+\delta w}\rs,\bar \alpha\rs)- \ell_1(s,\xi_{\bar u+\delta w}\rs, \alpha^\delta\rs)|\,ds\\
			&\mineq \eps C\int_t^\infty k_1(s)\psi(s)\,ds\mineq \eps C\norm{k_1}_{2,(t,\infty)}\norm{\psi}_{2,(t,\infty)} \doteq \eps \hat c.
		}
		From \cref{stimaparte1} we get	$J_{\bar \alpha}(t,x,\bar u+\delta w)-J_{\bar \alpha}(t,x,\bar u)\mageq J(t,x,\bar u+\delta w)-J(t,x,\bar u)-\hat c\eps$ for all $\delta\in (0,\tilde \delta_\eps).$
		Hence, dividing by $\delta$ and passing to the limit as $\delta\ra 0$,
		\eee{
			D_u J_{\bar \alpha}(t,x,\bar u)( w)\mageq D_u J(t,x,\bar u)( w)-\hat c \eps.
		}
		Since $\eps$ and $w$ are arbitrary, the proof is complete.
	\end{proof}
	\begin{remark}\label{remark_lip_cont}
The closeness of solutions estimate \rif{tubo} assumed in the statement of Proposition \ref{prop_W_uguale_sup_W_alpha} is satisfied for linear systems \rif{sistemacontrollo} with an asymptotically stable equilibrium point $x\in \ttnn{int }\Omega$ for input $u=0$ and $f_1$ totally bounded (cfr. \cite[Theorem 3.10]{khalil2002nonlinear}).
\end{remark}

	\section{Optimal synthesis}\label{optimal_synthesis}
	Let $h:\bb R^n\ra \bb R^n$ be a diffeomorphism, i.e., it is bijective and continuously differentiable with its inverse. In this section we give feedback laws for the optimal control problem \cref{W_alpha}, with dynamics and Lagrangian given by:  for all $s\in \rzero$, $x\in \bb R^n$, $u\in \bb R^m$, and $\alpha\mageq 0$
	\equazioneref{f_sezione_4}{
		&f_0(s,x)=\nabla h(x)^{-1}A\rs h(x),\quad f_1(s,x)=\nabla h(x)^{-1}B\rs\\
		&\ell  (s,\xi,u,\alpha)= \ps{h(\xi)}{Q(s,\alpha)h(\xi)} +\ps{u}{R u} -b(\alpha),
	}
	where $A:\bb R\ra \bb R^{n\times n}$, $B:\bb R\ra \bb R^{n\times m}$, $Q:\bb R\times \rzero \ra \bb R^{n\times n}$, $b:\bb R\ra \rzero$, and $R\in \bb R^{n\times n}$ are given.
	
	
	We consider  the following assumptions:
	\begin{hypothesis}\label{ipotesi_H_primo_primo}[\cref{ipotesi_H_primo_primo}]
		\begin{enumerate}[label=(\roman*)]
			\item $A, B,$ and $b$ are continuous;
			
			\item $R= \frac{1}{2} I_n$ and there exist $K\in L^1(\rpiu;\rpiu)$ and $a:\bb R\ra \rzero$ continuous such that $Q(s,\alpha)= (\frac{1}{2} K(s)+a(\alpha))I_n$.
		\end{enumerate}
	\end{hypothesis}
	%
	%
	
	%
	%
	%
	%
	%
	
	For any $\alpha\in \ccal A$ we denote by $P^\alpha_T\ccd$ the solution of the Riccati differential system
	\sistemaref{sistema_di_Riccati}{
		-{P}'=A^\star P+P A- P B R^{-1}B^\star P+Q^\alpha  & \ttnn{a.e. on }[t,T]\\
		P(T)=0,
	}
	where we put $Q^{\alpha}(s)\doteq Q(s,\alpha(s))$.
	
	The following result is well known ( \cite[Chapter 1, part 3]{DaPrato:2006:RepresentationControl} and \cite[Chapter 3]{andersonmoore1971linearoptimalcontrol}).
	\begin{lemma}[\cite{DaPrato:2006:RepresentationControl,andersonmoore1971linearoptimalcontrol}]\label{soluzione_Riccati_i_ii}
		Assume \cref{ipotesi_H_primo_primo}. Let $\alpha\in \ccal A$ and $(t,x)\in \ttnn{dom }W^\alpha$. Then the following holds:
		\enurom{
			\item $P^\alpha_T\in C([t,T];\bb R^{n\times n})\cap C^1((t,T);\bb R^{n\times n})$ and $P^\alpha_T(s)$ is positive definite for all $s\in[t,T]$ and all $T>t$;
			
			\item\label{Riccati_P_orizzonte_infinito_soluzione} for all $s\mageq t$ the limit $P^\alpha\rs\doteq \lim_{T\ra\infty}P^\alpha_T(s)$ exists, is positive definite, and it solves the Riccati differential equation
			\equazioneref{Riccati_inf_hor}{
				-{P}'=A^\star P+P A- P BR^{-1} B^\star P+Q^\alpha \quad  \ttnn{a.e. on }[t,\infty).
			}
			Such solution  is also called \textit{minimal (or stabilizing) solution} of the Riccati equation \cref{Riccati_inf_hor}.
		}	
	\end{lemma}
	In the following, for any $\alpha\in \ccal A$, $[t,T]\subset \rzero$, and $x\in \Omega$ we denote by $\xi^\alpha\ccd$ the solution of the Cauchy problem
	\sistemaref{riccatisystem}{
		\xi'\rs=\nabla h(\xi\rs)^{-1} \Gamma^\alpha\rs h(\xi\rs)&s\in [t,T]\ttnn{ a.e.}\\
		\xi(t)=x,
	}
	where $\Gamma^\alpha\rs\doteq A\rs -B\rs B^\star\rs P^\alpha(s)$.
	
	\begin{theorem}\label{theorem1}
		Assume \cref{ipotesi_H_primo_primo} and let $\alpha\in \ccal A$. Suppose that
		\equazioneref{IPCfull_su_tutto_rzero}{
			\Gamma^\alpha\rs h(x)\in \nabla h(x)^{-\star}\,(\ttnn{int }T_\Omega(x))\quad \forall s\in \rzero,\, \forall x\in \partial \Omega.
		}	
		Then, the value function of the auxiliary control problem \cref{W_alpha} satisfies for all $(t,x)\in \ttnn{dom }W^\alpha$
		\equazioneref{W_uguale_P_piu_int_onInfiniteHorizon}{
			W^\alpha(t,x)=\ps{h(x)}{P^\alpha\rt h(x)}-\int_t^\infty b(\alpha\rs)ds.
		}
		Moreover, the optimal input $u^\alpha_\infty\ccd$ satisfies the feedback law
		\equazioneref{feedbaklawinfty}{
			u^\alpha_\infty\rs=-B^\star\rs P^\alpha\rs h(\xi^\alpha\rs)
		}
		for a.e. $s\mageq t$, where $\xi^\alpha\ccd$ solves the Cauchy problem
		\sistemanoref{
			\xi'\rs=\nabla h(\xi\rs)^{-1} \Gamma^\alpha\rs h(\xi\rs)&s\in [t,\infty)\ttnn{ a.e.}\\
			\xi(t)=x.
		}
		
	\end{theorem}
	
	\begin{corollary}\label{corollary_main_theo}
		Assume \cref{ipotesi_H_primo_primo}, the set-valued map in \cref{set_valued_map_Lambda} has single valued images, and there exists a unique solution $( P^*, \xi^*)$ of
		\sistemanoref{
			-{P}'\rs=A^\star\rs P\rs+P\rs A\rs- P \rs B\rs R^{-1}\rs B^\star \rs P\rs+ Q^*\rs& s\in[t,\infty)\ttnn{ a.e.}\\
			\xi'\rs=\nabla h(\xi\rs)^{-1}  \Gamma^*\rs h(\xi\rs)& s\in[t,\infty)\ttnn{ a.e.}\\
			\xi(t)=x,
		}
		where $(t,x)\in \ttnn{dom }W$, ${ \alpha^*} \rs := \ttnn{argmax}_{\beta\mageq 0} \, (a(\beta)g( \xi^*\rs)-b(\beta))$, $Q^*(s):=Q^{ \alpha^*}(s)$, and $ \Gamma^*\rs:=\Gamma^{ \alpha^*}\rs$ for all $s\mageq t$.
		
		%
		
		Then, if \cref{IPCfull_su_tutto_rzero} holds, the optimal input $u^*_\infty\ccd$ for the state constrained control problem \cref{problema_W_inf} at the initial datum $(t,x)$ is given, for all $s\mageq t$, by
		\eee{
			u^*_\infty\rs=-B^\star\rs  P^*{{ }}\rs h( \xi^*{}\rs),
		}
		and
		\eee{
			W(t,x)=\sup_{\alpha\in \ccal A} W^\alpha(t,x).
		}
		\begin{remark}
			\
			
			\enualp{
				\item We notice that condition \cref{IPCfull_su_tutto_rzero} involves, implicitly, the stabilizing solution of the Riccati equation \cref{Riccati_inf_hor} given, accordingly to \cref{soluzione_Riccati_i_ii}-\ref{Riccati_P_orizzonte_infinito_soluzione}, by the pointwise limit of a sequence of Riccati solutions on increasing time intervals. Moreover, it reduces to an i.p.c. on the vector field $(s,x)\mapsto\Gamma^\alpha\rs x$ when $h$ is the identity (see \cref{remark1}). 
				More precisely, under the assumption that $\Omega$ is the closure of an open domain with smooth boundary, \cref{IPCfull_su_tutto_rzero} is as follows: $\forall x\in \partial \Omega$ and $\forall s\in \rzero$,  $\ps{\Gamma^\alpha(s)x}{n(x)}<0$,	where $n(x)$ denotes the exterior unit normal to $\Omega$ at $x$.
				
				
				
				\item Assuming more regularity on problem data, \cref{IPCfull_su_tutto_rzero} provides a neighboring feasible trajectories result  (see \Cref{lemma_appendice} in Appendix).
				
				\item We point out that, although the solutions of the system \cref{sistemacontrollo}, with $f_0$ and $f_1$ as in \cref{f_sezione_4}, are the same of the system $z'\rs=A\rs z\rs+B\rs u\rs$ under the trasformation $\xi=h^{-1}(z)$,
				the verification of the i.p.c. imposed in \Cref{theorem1} and \Cref{corollary_main_theo} is challenging for the set $h(\Omega)$.
				
			}

			%
			%
		\end{remark}
		
		
	\end{corollary}
	
	\subsection{Proofs}
	In this section we provide proofs of \cref{theorem1} and \cref{corollary_main_theo}. We first show some intermediate results.
	
	\begin{lemma}\label{lemmaRiccatifeasible}
		Assume \cref{ipotesi_H_primo_primo}. Let $P\in C([t,T];\bb R^{n\times n})$ with values in the set of all symmetric positive definite matrices, and suppose the following inward pointing condition on $[t,T]\subset \rzero$ holds:
		\equazioneref{IPCfull}{
			(A\rs -B\rs B^\star\rs P(s)) h(x)\in \nabla h(x)^{-\star}\,(\ttnn{int }T_\Omega(x))
		}
		for all $ s\in [t,T]$ and $x\in \partial \Omega$. Then, for any $x\in  \Omega$, the trajectory $\xi\ccd$ solving $\ttnn{a.e. on }[t,T]$
		\sistemaref{tesi_1_traiettoria_riccati}{
			\xi'=\nabla h(\xi)^{-1} (A\rs -B\rs B^\star\rs P(s)) h(\xi)\\
			\xi(t)=x,
		}
		is feasible.


	\end{lemma}
	
	\begin{proof}
		Arguing in analogous way as in \cref{proplscW} and considering the single-valued map given by
		\eee{
			(s,x)\rightsquigarrow\{\hat f(s,x)\doteq \nabla h(x)^{-1}(A\rs -B\rs B^\star\rs P(s)) h(x)\},
		}
		we conclude that $\xi\ccd$ solving \cref{tesi_1_traiettoria_riccati} is feasible.
	\end{proof}
	%

	\begin{lemma}\label{lemmaWT}
		Assume \cref{ipotesi_H_primo_primo}. Let $\alpha\in \ccal A$ and $(t,x)\in \ttnn{dom }W^\alpha$. If \cref{IPCfull} holds on $[t,T]$ then
		\eee{
			W_T^\alpha(t,x)=\ps{h(x)}{P^\alpha_T\rt h(x)}-\int_t^Tb(\alpha\rs)ds,
		}
		where $W_T^\alpha$ is the value function of the following state constrained Bolza problem
		\eee{
			&{\rm  minimize } \; \int_{t}^{T} \ps{h(\xi\rs)}{Q(s,\alpha\rs)h(\xi\rs)} +\ps{u\rs}{R u\rs} -b(\alpha\rs)\,ds
		}
		over all feasible trajectory-control pairs $(\xi\ccd,u\ccd)$ starting from $(t,x)$.

	\end{lemma}
	\begin{proof}
		Put $P_T\ccd=P^\alpha_T\ccd$ and define the Hamiltonian $H^\alpha:\bb R\times\bb R^n\times \bb R^n\ra \bb R$
		\eee{
			H^\alpha(s, x, p)=\inf _{u \in \mathbb{R}^{m}}\{\langle p,f(s,x,u)\rangle+\ell(s,x,u,\alpha\rs) \}.		
		}	
		Notice that
		\eee{
			H^\alpha(s, x, p)&= \langle p, \nabla h(x)^{-1}A\rs h(x)\rangle-\dfrac{1}{2}\left\langle p, \nabla h(x)^{-1}B\rs B^\star\rs \nabla h(x)^{-\star} p\right\rangle\\
			&\quad +\ps{h(x)}{Q(s,\alpha)h(x)}-{b(\alpha\rs)}.
		}
		Define $V^\alpha(s,x)\doteq \ps{h(x)}{P_T\rs h(x)}-\int_s^T b(\alpha(\tau))d\tau$ for all $(s,x)\in [t,T]\times \bb R^n$. Then for a.e. $s\in [t,T]$ and all $x\in \bb R^n$
		\eee{
			\frac{\partial V^\alpha}{\partial s}(s, x)&=\left\langle h(x), {P_T'}\rs h(x)\right\rangle+b(\alpha(s)) \\
			&=-\langle h(x),(A^\star\rs P_T\rs+P_T\rs A\rs- 2P_T\rs B\rs B^\star\rs P_T\rs+Q^\alpha(s)) h(x)\rangle\\
			&\qquad+{b(\alpha\rs)},\\
			\nabla_{x} V^\alpha(s, x)&=2\nabla h(x)^{\star} P_T\rs h(x).
		}
		It follows that $-\frac{\partial V^\alpha}{\partial s}(s, x)=H^{\alpha}\left(s, x, \nabla_{x} V^\alpha(s, x)\right)$ for a.e. $s\in [t,T]$ and all $x\in \bb R^n$.	Fix $x\in \Omega$ and let $u\in \ccal U(t,x)$. We have for a.e. $s\in [t,T]$
		\eee{
			0&\mineq \frac{\partial V^\alpha}{\partial s}(s, \xi_u\rs)+\ps{ \nabla_{x} V^\alpha(s, \xi_u\rs)}{\nabla h(\xi_u\rs)^{-1}A\rs h(\xi_u\rs)+\nabla h(\xi_u\rs)^{-1}B\rs u\rs}\\
			&\qquad+\frac{1}{2}\left(K(s)+2a\left(\alpha\rs\right)\right)|h(\xi_u\rs)|^{2}+\dfrac{1}{2}|u\rs|^2-{b(\alpha\rs)}.
		}	
		So,	$-\int_t^T \frac{\ttnn{d}}{\ttnn{d}s}V^\alpha(s, \xi_u\rs) ds\mineq \int_t^T\ell  (s,\xi_u\rs,u\rs,\alpha\rs)ds$, and, from \cref{sistema_di_Riccati}, we get $V^\alpha(t,x)\mineq$ $\int_t^T\ell  (s,\xi_u\rs,u\rs,\alpha\rs)ds$. Since $u\ccd$ is arbitrary it follows $V^\alpha(t,x)\mineq W_T(t,x)$.
		Now, the control defined by
		\equazioneref{controlRiccati}{
			u^\alpha\rs\doteq -B^\star\rs P_T\rs h(\xi^\alpha\rs)
		}
		is optimal for $H^\alpha(s,\xi\rs,\nabla_{x} V^\alpha(s, x))$ for a.e. $s\in [t,T]$. So, applying \cref{lemmaRiccatifeasible}, the trajectory $ \xi^\alpha\ccd$ is feasible. Thus
		\eee{
			-V^\alpha(T, \xi^\alpha(T))+V^\alpha(t,x)=\int_t^T\ell  (s, \xi^\alpha\rs, u^\alpha\rs,\alpha\rs)ds= V^\alpha(t,x).
		}
		We conclude $\inf_{u\in \ccal U(t,x)} \int_t^T\ell  (s,\xi_u\rs,u\rs,\alpha\rs)ds\mineq V^\alpha(t,x)$, and so $(iii)$ is proved.		
	\end{proof}
	
	Next, we give a proof of \cref{theorem1}.
	
	\begin{proof}[Proof of \cref{theorem1}]	Let $\alpha\in \ccal A$ and $(t,x)\in \ttnn{dom }W^\alpha$.
		
		Since there exists $u\in \ccal U(t,x)$ such that $\lim_{T\ra\infty} \int_t^T \ell  (s,\xi_u(s) ,u (s),\alpha(s) )\,ds<\infty$, we have that the limit $\lim_{T\ra\infty} \int_t^T b(\alpha\rs)ds$ exists and is finite. Then, from \cref{lemmaWT},
		\eee{
			\int_t^T \ell  (s,\xi_u(s) ,u (s),\alpha(s) )\,ds\mageq \ps{h(x)}{P^\alpha_T(t)h(x)}-\int_t^Tb(\alpha\rs)ds
		}
		for all $T>t$ and all $u\in \ccal U(t,x)$.	Since $u\ccd$ is arbitrary and applying \cref{lemmaWT}, passing to the limit as $T\ra+\infty$ we get $W^\alpha(t,x)\mageq \ps{h(x)}{P^\alpha\rt h(x)}-\int_t^\infty b(\alpha\rs)ds$.
		
		Now, let $T_j\uparrow \infty$. Applying \cref{lemmaRiccatifeasible}, denote for all $j\in \bb N$ by $\xi^\alpha_j \ccd$ the Riccati feasible trajectory associated to the control on $[t,T_j]$ defined by $	u^\alpha_j\rs\doteq -B^\star\rs P^\alpha_{T_j}\rs h(\xi^\alpha_j\rs)$.
		Now, for all $T>t$, $P_j\doteq  P^\alpha_{T_j}$ are uniformly bounded on $[t,T]$ whenever $T_j>T$. So, for all $T>t$, arguing in analogous way as in \cref{proplscW}, we have that $\xi^\alpha_j \ccd$ are equintegrable, equicontinuous, and equibounded on $[t,T]$ for all large $j$. From the Ascoli-Arzel\`{a} and the Dunford-Pettis theorems ( \cite{brezis2010functional}), and arguing as in \cref{proplscW}, we conclude that there exists a feasible absolutely continuous trajectory $\xi^1\ccd$ solving \cref{riccatisystem} on $[t,t+1]$, starting from $x$, and $\xi^\alpha_j \ra\xi^1$ uniformly on $[t,t+1]$. Consider now the interval $[t,t+2]$. Arguing as above, passing to subsequences and keeping the same notation, we may conclude that there exists $\xi^2\ccd$ solving \cref{riccatisystem} on $[t,t+2]$ and starting from $x$ such that $\xi^2|_{[t,t+1]}=\xi^1\ccd$ and $\xi^\alpha_j\ra \xi^2$ uniformly on $[t,t+2]$. Using a diagonal argument, passing to subsequences and keeping the same notation, we conclude that there exists a feasible trajectory $\xi^\alpha\ccd$ solving \cref{riccatisystem} on $[t,\infty)$, starting from $x$, and $\xi^\alpha_j \ra\xi^\alpha$ uniformly on compact intervals. Denote the limit $u^\alpha_\infty\rs\doteq\lim_{j\ra \infty}-B^\star\rs P_j\rs h(\xi^\alpha_j\rs)$ for all $s\mageq t$.   We have for all large $j\in \bb N$
		\eee{
			\ps{h(x)}{P_j(t)h(x)}-\int_t^{T_j}b(\alpha\rs)ds&=\int_t^{T_j} \ell  (s,\xi^\alpha_j(s) , u^\alpha_j(s),\alpha(s) )\,ds\\
			&\mageq \int_t^{T} \ell  (s,\xi^\alpha_j(s) , u^\alpha_j(s),\alpha(s) )\,ds.
		}
		Then, passing to the limit as $j\ra \infty$ and using Fatou's Lemma, we get $\ps{h(x)}{P^\alpha(t)h(x)}-\int_t^{\infty}b(\alpha\rs)ds\mageq \int_t^{T} \ell  (s,\xi^\alpha(s) , u^\alpha_\infty(s),\alpha(s) )\,ds$. By arbitrariness of $T$, we finally get \cref{W_uguale_P_piu_int_onInfiniteHorizon}.
	\end{proof}
	
	\begin{proof}[Proof of \cref{corollary_main_theo}]
		Applying the results in \cite[Chapter 14]{andersonmoore1971linearoptimalcontrol} on the stabilizing solution of the Riccati equation \cref{Riccati_inf_hor} and \cref{remark_lip_cont}, the Lipschitz continuity on the initial datum given in \cref{tubo} is satisfied along the optimal trajectory. Hence the conclusions follows from \Cref{theorem1}, \Cref{prop_W_uguale_sup_W_alpha}, \cref{Lemmachiave}-\ref{lemma_statement_alpha_delta_meas_selection}, and the proof of \cite[Theorem 5.14]{dowercantonimcene2019gamerepbarr}.
	\end{proof}

	
	\subsection{A geometric condition}
	
	Next we provide a sufficient geometric condition to recover the i.p.c. when the matrix $A$ is time independent.
	
	\begin{proposition}\label{corollario_1}
		Assume \cref{ipotesi_H_primo_primo} with $B\in L^\infty(\rzero;\bb R^{n\times m})$ and $A$ time independent. Sup-pose there exists $\delta>0$ such that
		\equazioneref{ipotesi_geometrica}{
			h(x)-\delta \nabla h(x)^{-\star}(N^1_\Omega(x))\subset \delta \nabla h(x)^{-\star} (\ttnn{int }\bb B)\quad  \forall x\in \partial \Omega.
		}	
		Then for any $\alpha\in \graffe{\beta\,:\, \ttnn{dom }W^\beta \neq \emptyset}$ there exists a constant $\bar \gamma>0$ such that, if $A$ is $\gamma$-negative definite with $\gamma>\bar \gamma$, all conclusions of \cref{theorem1} holds true.
		
	\end{proposition}
	\begin{proof}
		First of all, we notice that, replacing $h\ccd$ with $\sqrt{\delta} h\ccd$, the solutions of \cref{riccatisystem} are the same. So, we formally denote $h\ccd$ the function given by $x\mapsto \sqrt{\delta} h(x)$. From \cref{ipotesi_geometrica} and since $\nabla h(x)^{-\star}$ has full rank, it follows that $|h(x)-\nabla h(x)^{-\star}n|<|\nabla h(x)^{-\star}n|$ for any $n\in N^1_\Omega(x)$ and $x\in \partial\Omega$. Since $|\nabla h(x)^{-\star}n|\neq 0$, we have that $|\nabla h(x)^{-\star}n||h(x)-\nabla h(x)^{-\star}n|-|\nabla h(x)^{-\star}n|^2<0$ for any $x\in \partial \Omega$ and $n\in N^1_\Omega(x)$. From the compactness of $\Omega$ and $N_\Omega(x)\cap \partial \bb B$ and the continuity of $h\ccd$ and $\nabla h(\cdot)^{-1}$, it follows that there exists $\rho=\rho_{h,\Omega}>0$ satisfying $|\nabla h(x)^{-\star}n||h(x)-\nabla h(x)^{-\star}n|-|\nabla h(x)^{-\star}n|^2\mineq -\rho$ for all $x\in \partial \Omega$ and $\forall n\in N^1_\Omega(x)$.
		Moreover, there exists a constant $\theta=\theta_{h,\Omega}>0$ such that $|\nabla h(x)^{-\star} n(x)||h(x)-\nabla h(x)^{-\star}n|\mineq \theta$ for all $x\in\partial \Omega$ and $n\in N^1_\Omega(x)$. Now, fix $\alpha\in \ccal A$ such that $\ttnn{dom }W^\alpha\neq \emptyset$ and denote
		\eee{
			\bar \gamma\doteq\rho^{-1}\theta \norm{B}_{\infty,\rzero}^2\norm{C_\alpha}_{2,(0,\infty)}^2,
		}
		where $C_\alpha\ccd\doteq\sqrt{Q(\cdot,\alpha\ccd)}$. Assume that $A$ is $\gamma$-negative definite with $\gamma>\bar \gamma$ and let $P=P^\alpha$ be as in \cref{soluzione_Riccati_i_ii}-\ref{Riccati_P_orizzonte_infinito_soluzione}. Since $P\rs$ and $B\rs B^\star\rs$ are positive definite for all $s$, $BB^\star P$ is positive definite. Hence, the matrix $A-B\rs B^\star\rs P\rs$ is $\gamma$-negative definite. Hence for all $n(x)\in N_\Omega(x)\cap \partial \bb B$ and using the Cauchy-Schwarz inequality we have
		\equazioneref{stima1_lemmaWT}{
			&\ps{\nabla h(x)^{-1}\Gamma^\alpha\rs h(x)}{ n(x)}\\
			&= \ps{\Gamma^\alpha\rs\nabla h(x)^{-\star} n(x)}{\nabla h(x)^{-\star} n(x)}+\ps{\Gamma^\alpha\rs(h(x)-\nabla h(x)^{-\star} n(x))}{\nabla h(x)^{-\star} n(x)}\\
			&\mineq -\gamma|\nabla h(x)^{-\star} n(x)|^2+\norm{{\Gamma^\alpha\rs}} |\nabla h(x)^{-\star} n(x)||h(x)-\nabla h(x)^{-\star} n(x)|.
		}
		Now, notice that for any $T\mageq t\mageq 0$, the solution $P_T\ccd=P^\alpha_T\ccd$ of the Riccati differential system \cref{sistema_di_Riccati} on $[t,T]$ satisfies, for all $x\in \Omega$ and $s\in [t,T]$ ( \cite{DaPrato:2006:RepresentationControl}),
		\eee{
			P_T(s) x=
			\int_{s}^{T} e^{\tau A^{\star}} {C_\alpha}^{\star}(\tau) C_\alpha (\tau) e^{\tau A} x d \tau-2\int_{t}^{T} e^{(T-\tau) A^{\star} } P_T(\tau) B(\tau) B^{\star}(\tau) P_T(\tau) e^{(T-\tau) A} x d \tau.
		}
		Thus, since $\norm{e^{\tau A}}\mineq e^{-\tau \gamma}$ for all $\tau \mageq 0$, we have $\ps{P_T\rs x}{x}\mineq  \int_s^T \norm{C_\alpha(\tau)}^2\norm{e^{\tau A}}^2|x|^2d\tau $ $\mineq \int_s^T \norm{C_\alpha(\tau)}^2|x|^2d\tau\mineq\norm{C_\alpha}_{2,(0,\infty)}^2|x|^2$ for all $s\in [t,T]$ and all $T\mageq t\mageq 0$. Hence, passing to the limit as $T\ra \infty$ and applying \cref{lemmaWT}, $\norm{P\rs}\mineq \norm{C_\alpha}_{2,(0,\infty)}^2$. It follows that
		\equazioneref{stima_gamma}{
			\norm{\Gamma^\alpha\rs}\mineq \gamma +\norm{B\rs}^2\norm{C_\alpha}_{2,(0,\infty)}^2\quad \forall s\mageq 0.
		}	
		So, using \cref{stima1_lemmaWT} and \cref{stima_gamma}, we conclude that for all $x\in\partial \Omega$, $s\mageq 0$, and $n\in N^1_\Omega(x)$
		\eee{
			\ps{\nabla h(x)^{-1}\Gamma^\alpha\rs h(x)}{ n(x)}&\mineq (|\nabla h(x)^{-\star}n||h(x)-\nabla h(x)^{-\star}n|-|\nabla h(x)^{-\star}n|^2)\gamma\\
			&\qquad +\norm{B(s)}^2\norm{C}_{2,(0,\infty)}|\nabla h(x)^{-\star} n(x)||h(x)-\nabla h(x)^{-\star}n|\\
			&\mineq -\rho \gamma + \theta\norm{B(s)}^2\norm{C_\alpha}_{2,(0,\infty)}^2.
		}
		Then \cref{IPCfull_su_tutto_rzero} is satisfied and the conclusion follows from \cref{theorem1}.
	\end{proof}
	
	\begin{corollary}
		Assume the assumptions of \cref{corollario_1} with $h=\ttnn{id}$ and there exists $r>0$ such that $B(0,r)\subset \Omega$. Then, for any $\alpha\in \graffe{\beta\,:\, \ttnn{dom }W^\beta \neq \emptyset}$ and any $(t,x)\in (\rzero\times B(0,r))\cap \ttnn{dom }W^\alpha$, the representation \cref{W_uguale_P_piu_int_onInfiniteHorizon} and the feedback laws \cref{feedbaklawinfty} holds.
	\end{corollary}
	\begin{proof}
		The proof follows immediately from \cref{corollario_1} and since \cref{ipotesi_geometrica} is satisfied with $\delta=r$ in which formally $\Omega$ is replaced by $B(0,r)$.
	\end{proof}
	
	%
	
	\appendix \section*{Appendix}
	\renewcommand{\theequation}{A.\arabic{equation}}
	
	\begin{applemma}\label{lemma_appendice}
		Assume the assumptions of \cref{lemmaRiccatifeasible} and moreover that $A$ and $B$ are locally absolutely continuous and $B\in L^\infty(t,T;\bb R^{n\times m})$. Then there exists $\beta>0$ such that for all $x\in \Omega$ and $\sigma>0$ we can find $\xi^\sigma\ccd$ feasible for \cref{sistemacontrollo} on $[t,T]$, with $\xi^\sigma(t)=x$, satisfying
		\equazioneref{lemma_statement_NFT}{
			\norm{\xi-\xi^\sigma}_{\infty,[t,T]}\mineq \beta \sigma,\qquad \xi^\sigma\ccd\subset \ttnn{int } \Omega.
		}
	\end{applemma}
	\begin{proof}
		We take the same notation as in the proof of \Cref{lemmaRiccatifeasible}. We show the following claim: there exist \(\varepsilon>0\) and \(\eta>0\) satisfying for all $ (s,x) \in [t, T] \times(\partial \Omega+\eta \mathbb{B})  \cap \Omega$ and all $ y \in(x+\varepsilon \mathbb{B}) \cap \Omega$
		\equazioneref{claim}{
			y+[0, \varepsilon](\hat f(s,x)+\varepsilon \mathbb{B}) \subset \Omega.
		}
		Notice that for any \((s,x) \in[t, T] \times \partial \Omega\) and from	
		the characterization of the interior of the Clarke tangent cone (cfr \cite{aubin2009set}) , we can find	\(\varepsilon \in(0,1)\) such that	\(y+[0, \varepsilon](\hat f(s,x)+2 \varepsilon \mathbb{B}) \subset \Omega \) for all \(y \in(x+2 \varepsilon \mathbb{B}) \cap \Omega\).
		Now take any \(\tilde y \in\left(\tilde x+\varepsilon \mathbb{B}\right) \cap \Omega .\) Then, since \(\tilde x+\varepsilon \mathbb{B} \subset x+2 \varepsilon \mathbb{B}\) and \(|\hat f(\tilde t,\tilde x) |\mineq |\hat f(s,x)|+\varepsilon ,\) we may conclude \( \tilde y+[0, \varepsilon]\left(\hat f(\tilde t,\tilde x)+\varepsilon \mathbb{B}\right) \subset \Omega \) for all \(\tilde y \in\left(\tilde x+\varepsilon \mathbb{B}\right) \cap \Omega\). So, we have shown that for any \((s,x) \in[t, T] \times \partial \Omega\) there exist \( \varepsilon_{s,x} \in(0,1)\) and	\(\delta_{s,x} \in\left(0, \varepsilon_{s,x}\right]\) such that, given any \(\left(\tilde t, \tilde x\right) \in\left((s,x)+\delta_{s,x} \mathbb{B}\right) \cap([t, T] \times \Omega),\)
		\eee{
			\graffe{\tilde y+\left[0, \varepsilon_{s,x}\right]\left(\hat f(\tilde t,\tilde x)+\varepsilon_{s,x }\bb B\right)\,:\,y \in\left(\tilde x+\varepsilon_{s,x }\mathbb{B}\right) \cap \Omega}\subset \Omega.
		}
		Using a compactness argument, we conclude that there exist \(\left(t_{i}, x_{i}\right) \in[t, T] \times\partial \Omega\) and $0<\delta_i<\varepsilon_i$, for \(i\in \{1,...,N\}\), such that $ [t, T] \times \partial \Omega \subset \bigcup_{i=1}^{N}\left(\left(t_{i}, x_{i}\right)+\delta_{i} \ttnn{int }\mathbb{B}\right)$, and, for any \(\left(\tilde t, \tilde x\right) \in\left(\left(t_{i}, x_{i}\right)+\delta_{i} \mathbb{B}\right) \cap([S, T] \times \Omega),\)
		\eee{
			\tilde y+\left[0, \varepsilon_{i}\right](\hat f(\tilde t,\tilde x)+\varepsilon_{i} \mathbb{B}) \subset \Omega\quad  \forall \tilde y \in\left(\tilde x+\varepsilon_{i} \mathbb{B}\right) \cap \Omega.
		}
		Notice also that there exists \(\eta \in\left(0, \min _{i} \delta_{i}\right)\) satisfying
		$  [t, T] \times(\partial \Omega+\eta \mathbb{B})   \subset \bigcup_{i=1}^{N}((t_{i}, x_{i})$ $+\delta_{i} \ttnn{int }{B}) $ (otherwise we could find a sequence of points $\left(s_{j}, y_{j}\right) \notin \bigcup_{i=1}^{N}\left(\left(t_{i}, x_{i}\right)+\delta_{i} \mathbb{B}\right)$ such that \(\left(s_{j}, y_{j}\right) \rightarrow(s, y) \in[t, T] \times \partial \Omega\)) . The claim \cref{claim} just follows taking $\varepsilon=\min _{i} \varepsilon_{i}$. Consider now the following differential inclusion
		\equazioneref{sistemacontrollomineqK}{
			\xi'\rs\in G(s,\xi\rs) \quad s\in[t,T]\ttnn{ a.e.},\quad	\xi(t)=x,
		}
		where $G(s,x)\doteq \{ f_0(s,x)+f_1(s,x)u\,:\, |u|\mineq \norm{B^\star P}_{\infty,[t,T]}\norm{h}_{\infty,\Omega}\}$ and $f_0,f_1$ are as in \cref{f_sezione_4}.	Notice that $\hat f(s,x)\in G(s,x)$ for any $(s,x)\in [t,T]\times \bb R^n$ and the trajectory $\xi\ccd$ is solution of \cref{sistemacontrollomineqK} with $ u\rs=-B^\star \rs P\rs h(\xi\rs)$. Moreover, $\sup\{|v|\,:\, v\in G(s,x), s\in [t,T],x\in (\partial \Omega+\bb B)\} =M<\infty $ and there exists $\lambda \in L^1([t,T];\rzero)$ such that $G(s,x)\subset G(s',x)+\int_s^{s'} \lambda(\tau)d\tau$ for all $S\mineq s<s'\mineq T$ and $x\in \Omega$.
		So, arguing in analogous way as in \cite[Theorem 1]{bascofrankowska2018lipschitz} and using \cref{claim}, we conclude that there exists $\beta>0$ (depending on the time interval $[t,T]$) such that for any $x\in \Omega$ and any $\sigma>0$ there a feasible trajectory $\xi^\sigma\ccd$, solving \cref{sistemacontrollomineqK} on $[t,T]$ and starting from $x$, that satisfies \cref{lemma_statement_NFT}. Hence the conclusion follows by applying the measurable selection theorem.
	\end{proof}

	\bibliographystyle{plain}
	\bibliography{BIBLIO_VBPD_twoplayer_rep_inf_hor}

\begin{thebibliography}{10}

\bibitem{andersonmoore1971linearoptimalcontrol}
B.~D.~O. Anderson and J.~B. Moore.
\newblock {\em Linear optimal control}.
\newblock Prentice-Hall, Inc., Englewood Cliffs, N.J., 1971.

\bibitem{Aubin:1991:viabilitytheory}
J.-P. Aubin.
\newblock {\em Viability Theory}.
\newblock Birkhauser Boston Inc., 1991.

\bibitem{aubin2009set}
J.-P. Aubin and H.~Frankowska.
\newblock {\em Set-valued analysis}.
\newblock Modern Birkh\"auser Classics. 2009.

\bibitem{bascofrankcann2017necessary}
V.~Basco, P.~Cannarsa, and H.~Frankowska.
\newblock Necessary conditions for infinite horizon optimal control problems
  with state constraints.
\newblock {\em Mathematical Control \& Related Fields}, 8(3\&4):535--555, 2018.

\bibitem{bascocannfrank2019semisubRiem}
V.~Basco, P.~Cannarsa, and H.~Frankowska.
\newblock Semiconcavity results and sensitivity relations for the
  sub-riemannian distance.
\newblock {\em Nonlinear Analysis}, 184:298--320, 2019.

\bibitem{bascofrankowska2018hjbe}
V.~Basco and H.~Frankowska.
\newblock {H}amilton-{J}acobi-{B}ellman equations with time-measurable data and
  infinite horizon.
\newblock {\em Nonlinear Differential Equations and Applications}, 26(1):7, Feb
  2019.

\bibitem{bascofrankowska2018lipschitz}
V.~Basco and H.~Frankowska.
\newblock Lipschitz continuity of the value function for the infinite horizon
  optimal control problem under state constraints.
\newblock {\em In "Trends in Control Theory and Partial Differential
  Equations", ed. F. Alabau-Boussouira et al., Springer INdAM Series}, 2019.

\bibitem{DaPrato:2006:RepresentationControl}
A.~Bensoussan, G.~Da~Prato, M.C. Delfour, and S.K. Mitter.
\newblock {\em Representation and Control of Infinite Dimensional Systems}.
\newblock Birkhauser, 2006.

\bibitem{brezis2010functional}
H.~Brezis.
\newblock {\em Functional analysis, Sobolev spaces and partial differential
  equations}.
\newblock Springer Science \& Business Media, 2010.

\bibitem{cannfrank2018valuefunction}
P.~Cannarsa and H.~Frankowska.
\newblock Value function, relaxation, and transversality conditions in infinite
  horizon optimal control.
\newblock {\em J. Math. Anal. Appl.}, 457(2):1188--1217, 2018.

\bibitem{Carlson:1987:IHO:38189}
D.~A. Carlson and A.~Haurie.
\newblock {\em Infinite Horizon Optimal Control: Theory and Applications}.
\newblock Springer-Verlag New York, Inc., 1987.

\bibitem{curtainprichard1978infinitelinear}
R.~F. Curtain and A.~J Pritchard.
\newblock {\em Infinite dimensional linear systems theory}.
\newblock Springer-Verlag Berlin, New York, 1978.

\bibitem{dapratoichikawa1987optimalcontrol}
G.~Da~Prato and A.~Ichikawa.
\newblock Optimal control of linear systems with almost periodic inputs.
\newblock {\em SIAM Journal on Control and Optimization}, 25(4):1007--1019,
  1987.

\bibitem{dapratoichikawa1988quadraticcontrollinearsystems}
G.~Da~Prato and A.~Ichikawa.
\newblock Quadratic control for linear periodic systems.
\newblock {\em Applied Mathematics and Optimization}, 18(1):39--66, 1988.

\bibitem{dowercantoni2017state}
P.~M. Dower and M.~Cantoni.
\newblock State constrained optimal control of linear time-varying systems.
\newblock In {\em 2017 IEEE 56th Annual Conference on Decision and Control
  (CDC)}, pages 1338--1343. IEEE, 2017.

\bibitem{dowercantonimcene2019gamerepbarr}
P.~M. Dower, W.~M. McEneaney, and M.~Cantoni.
\newblock Game representations for state constrained continuous time linear
  regulator problems.
\newblock {\em arXiv preprint arXiv:1904.05552}, 2019.

\bibitem{DMC2:16}
P.M. Dower, W.M. McEneaney, and M.~Cantoni.
\newblock A game representation for state constrained linear regulator
  problems.
\newblock In {\em proc. $55^{th}$ IEEE Conference on Decision \& Control (Las
  Vegas NV, USA)}, pages 1074--1079, 2016.

\bibitem{frankowskaplaskrze1995measviabth}
H.~Frankowska, S.~Plaskacz, and T.~Rze\.zuchowski.
\newblock Measurable viability theorems and the {H}amilton-{J}acobi-{B}ellman
  equation.
\newblock {\em J. Differential Equations}, 116(2):265--305, 1995.

\bibitem{khalil2002nonlinear}
H.~Khalil.
\newblock {\em Nonlinear systems}.
\newblock Prentice Hall, 1996.

\bibitem{rockafellar1970conjugate}
R.~T. Rockafellar.
\newblock Conjugate convex functions in optimal control and the calculus of
  variations.
\newblock {\em Journal of Mathematical Analysis and Applications},
  32(1):174--222, 1970.

\bibitem{rockafellar1974conjugate}
R.~T. Rockafellar.
\newblock Conjugate duality and optimization.
\newblock {\em Siam Regional Conf. Series in Applied Math.}, 16, 1974.

\bibitem{rockafellar2008linear}
R.~T. Rockafellar and R.~Goebel.
\newblock Linear-convex control and duality.
\newblock In {\em Geometric Control And Nonsmooth Analysis: In Honor of the
  73rd Birthday of H Hermes and of the 71st Birthday of RT Rockafellar}. World
  Scientific, 2008.

\bibitem{rockafellar2009variational}
R.~T. Rockafellar and R.~J.~B. Wets.
\newblock {\em Variational analysis}.
\newblock Springer-Verlag, Berlin, 1998.

\bibitem{seierstad1986optimal}
A.~Seierstad and K.~Syds{\ae}ter.
\newblock {\em Optimal control theory with economic applications}.
\newblock North-Holland Publishing Co., Amsterdam, 1987.

\bibitem{Soner86a}
H.~M. Soner.
\newblock Optimal control problems with state-space constraints {I}.
\newblock {\em SIAM J. Control Optim.}, 24:552--562, 1986.

\bibitem{vinterpappas1982amaximumprinciple}
R.~B. Vinter and G.~Pappas.
\newblock A maximum principle for nonsmooth optimal-control problems with state
  constraints.
\newblock {\em J. Math. Anal. Appl.}, 89(1):212--232, 1982.

\end{thebibliography}

\end{document}